\documentclass{amsart}
\usepackage{amsthm,amssymb,latexsym,amsmath}
\usepackage[all]{xypic}
\usepackage{cite}
\usepackage{stackengine}
\stackMath

\usepackage{tikz}

\newcommand{\catA}[1]{{\mathfrak A}}
\newcommand{\catI}[1]{{\mathfrak I}}
\newcommand{\catS}[1]{{\mathfrak S}}

\newcommand{\C}{\mathbb{C}}

\newcommand{\coker}{{\rm coker}}
\newcommand{\Hilb}{{\rm Hilb}}
\newcommand{\Quot}{\rm Quot}

\newcommand{\im}{{\rm Im}~}

\newcommand{\p}[1]{{\mathbb{P}^{#1}}}

\newcommand{\pn}{{\mathbb{P}^n}}

\newcommand{\N}{\mathbb{N}}

\DeclareMathOperator{\End}{End}

\DeclareMathOperator{\ho}{H}

\DeclareMathOperator{\Mat}

\newtheorem{theorem}{Theorem}[section]

\newtheorem{proposition}[theorem]{Proposition}
\newtheorem{lemma}[theorem]{Lemma}
\newtheorem{corollary}[theorem]{Corollary}

\newtheorem{definition}[theorem]{{\bf Definition}}

\begin{document}



\title[Quot schemes of points on affine spaces]{A note on the ADHM description of Quot schemes of points on affine spaces}

\author{Abdelmoubine A. Henni}
\author{Douglas M. Guimar\~aes}

\address{Universidade Federal de Santa Catarina \\
Departamento de Matem\'atica \\
Campus Universit\'ario Trindade\\
CEP 88.040-900 Florian\'apolis-SC, Brasil} 
\email{henni.amar@ufsc.br}

\address{Universidade Estadual de Campinas \\
Departamento de Matem\'atica  \\
Cidade Universitária "Zeferino Vaz" \\
CEP 13083-970  Campinas-SP \\
}
\email{ra209447@ime.unicamp.br}

\begin{abstract}
We give an ADHM description of the Quot scheme of points $\Quot_{\mathbb{C}^{n}}(c,r),$ of length $c$ and rank $r$ on affine spaces $\mathbb{C}^{n}$ which naturally extends both Baranovsky's representation of the punctual Quot scheme on a smooth surface and the Hilbert scheme of points on affine spaces $\mathbb{C}^{n},$ described by the first author and M. Jardim. Using results on the variety of commuting matrices, and combining them with our construction, we prove new properties concerning irreducibility and reducedness of  $\Quot_{\mathbb{C}^{n}}(c,r)$ and its punctual version $\Quot_{\mathbb{Y}}^{[p]}(c,r),$ where $p$ is a fixed point on a smooth affine variety $\mathbb{Y}.$ In this last case we also study a connectedness result, for some special cases of higher $r$ and $c.$

\end{abstract}



\maketitle
\tableofcontents


\section{Introduction}

For integers $r,c>0$ we consider the Quot scheme $\Quot_{\mathbb{Y}}(c,r)$ of zero-dimensional quotients $\lbrack \mathcal{O}_{\mathbb{Y}}^{r}\twoheadrightarrow\mathcal{Q}\rbrack,$ of length $c$ over a complex algebraic variety $\mathbb{Y}.$ This Quot scheme was intensely studied by various authors especially when $\mathbb{Y}$ is a smooth surface \cite{Bara,Elling,Li,Giesk}. In this case, $\Quot_{\mathbb{Y}}(c,r)$ is well known to be smooth and irreducible of dimension $c(r+1)$ \cite[\S II.6.A]{Huy}. Moreover, when $\mathbb{Y}=\mathbb{P}^{2},$ one can describe $\Quot_{\mathbb{P}^{2}}(c,r)$ in terms of linear data called  ADHM (Atiyah-Drinfel'd-Hitchin-Manin) data \cite{ADHM,Bara,HJM,HJ}, as it appears in Baranovsky's  proof of the irreducibility of the punctual Quot scheme \cite[Appendix A]{Bara}, where the support of the quotient sheaf $\mathcal{Q}$ is a single point. Such a description is a generalization of the Nakajima's representation of the Hilbert scheme of points on the complex plane \cite{N2}.

The ADHM data appears to be very useful in various situations when one is concerned with explicit constructions of (framed) instanton sheaves on various algebraic varieties, see \cite{Barth,donaldson1,OSS,HJM,Tikh1,Tikh2} and the references therein. In particular, on surfaces, when the rank is $1$ one gets those data for ideal sheaves of points. These algebraic data come from complexes called monads \cite[Ch.II \S3]{OSS}, \cite{HJM}, obtained in many cases from Beilinson's spectral sequence \cite[Ch.II \S3]{OSS} by using some homological characterizations. 

In \cite[\S3]{HJ} the first author and Jardim introduced complexes called {\em perfect extended monads}, as a generalization of the usual monads, in order to study the Hilbert scheme of points on affine varieties of higher dimensions. Such complexes also appear as a result of some suitable homological characterization combined to  Beilinson's spectral sequence. Furthermore, by looking at their automorphisms, one could get algebraic spaces of ADHM data that represent the Hilbert functor of points on affine varieties. This led us to prove that $\Hilb^{[c]}({\mathbb{C}^{3}})$ is irreducible for $c\le 10,$ \cite[Corollary 6.2]{HJ}. The result was achieved by connecting the ADHM data to some results concerning the irreducibility of the variety of $\mathcal{C}(n,c)$ of $n$ commuting $c\times c$ matrices \cite{G,HO,MT,OCV,S}. We remark that, for $\mathbb{C}^{3},$ the cases $c\le8$ were already known \cite{Fog,CEVV}. But our new proof allowed to extend those results to the cases $c=9$ and $c=10.$

In this work, we make use of the results obtained in \cite{HJ} in order to describe the Quot scheme of points in terms of perfect extended monads in addition to ADHM data. This leads us to prove some new results, namely;

\vspace{0.3cm}

{\bf First main result:} (Corollary \ref{small-irred} )

\begin{itemize}
\item[(i)] $\Quot_{\mathbb{C}^{3}}(c,r)$ is irreducible for $c\le 10,$ for any $r\geq 1.$

\item[(ii)] $\Quot_{\mathbb{C}^{n}}(c,r)$ is irreducible for $c\leq3$ and $\forall n\geq4,$ and $r\geq1.$ Moreover, if $c=2$ it is also reduced for every $n,r\geq 1.$

\end{itemize}

\vspace{0.3cm}

Moreover, by using results on nilpotent commuting matrices from \cite{Ngo, Ngo-Sivic} we obtain new irreducibility properties of the punctual Quot schemes over an affine variety $\mathbb{Y}$ of dimension $n,$ in the first two items bellow. The third item, bellow, is a direct consequence of our construction and the ADHM stability (Definition \ref{stability}):   

\vspace{0.2cm}
{\bf Second main result:} (Corollary \ref{quot:pt} \& Corollary \ref{cor-conecd})

\begin{itemize}
\item[(i)] $\Quot_{\mathbb{Y}}^{[p]}(c,r)$ is irreducible for $c\le 6,$ for any $r\geq 1$ and ${\rm dim}\mathbb{Y}=3,$

\item[(ii)] $\Quot_{\mathbb{Y}}^{[p]}(c,r)$ is irreducible for $c\leq 3 ,$ for any $n,r\geq 1.$ Moreover, for $c=2$ it is reduced of dimension $2r+n-3$ and for $c=3$ its dimension is $2n+3r-5.$

\item[(iii)] Let $n,c\in \N$ and $r\geq n$. Then $\Quot_{\mathbb{Y}}^{[p]}(c,r)$ is path connected.

\end{itemize}

In \cite{H}, Hartshorne proved that the Hilbert scheme is connected. Particularly, when the Hilbert polynomial is constant, this corresponds to the case $r=1,$ On the other hand, item ${\rm (iii)}$ is the most general result for higher $c$ and $r,$ and the only cases known to us are when a unique irreducible component exists.

The outline of this paper will be as the following: in section \ref{Matrix-Para} we give a brief review on the variety of $n$ commuting $c\times c$ matrices and ADHM data. Section \ref{ext-monads} is devoted to $l-$extended and perfect monads. In section \ref{z-cy} we show the three folded relation between quotients of copies of the structure sheaf, of $\mathbb{Y}$, wich are supported in zero dimension, perfect extended monads and their ADHM data. This will be useful in giving a proof of representability of the Quot functor under consideration. 
Althought the relation of bundles, ADHM data and representations of quivers is a part of a folklore \cite{BDJ,FJM,JP,N1,N2}, in Section \ref{quivers}, we show, in our special case, that the space of stable ADHM data and the space of $\vartheta-$representations of a special quiver, with relations over $\mathbb{Y},$ for some generic stability parameter $\vartheta,$ lead to the same quotient, namely $\Quot_{\mathbb{Y}}(c,r).$ 

Finally, in section \ref{irred} we prove the main results about irreducibility and connectedness.

\section{Commuting matrices and stable ADHM data}\label{Matrix-Para}

In this section we shall introduce the necessary material to our construction. Most of which is a generalization of the \cite[Section 2]{HJ} : let $V$ be a complex vector space of dimension $c$ and let $B_{0},B_{1},\ldots,B_{n-1}\in\End(V)$  be $n$ linear operators on $V$.

\begin{definition}
The variety $\mathcal{C}(n,c)$ of $n$ commuting linear operators on $V$ is the subvariety of $\End(V)^{\oplus n}$ given by
$$ \mathcal{C}(n,c)=\left\{(B_{0},B_{1},\ldots,B_{n-1})\in\End(V)^{\oplus n} ~|~  [B_i, B_j]=0, \, \forall i\neq j\right\} . $$
\end{definition}

The commutation relations can be thought of as a system of ${n\choose2}c^{2}$ homogeneous equations of degree $2$ in $nc^{2}$ variables. 


The variety $\mathcal{C}(n,c)$ have been extensively studied by various authors since a 1961 paper by Gerstenhaber \cite{G}, with special attention to its irreducibility. We recall here some of the main results:

\begin{theorem}\label{Mtz-Tau}
    $\mathcal{C}(2,c)$ is irreducible for every $c$.
\end{theorem}

\begin{theorem}\label{Meara}
    $\mathcal{C}(3,c)$ is irreducible for $c\leq 10$ and reducible for $c\geq 29.$
\end{theorem}

\begin{theorem}\label{Gersten}
 For $n\ge4$, $\mathcal{C}(n,c)$ is irreducible if and only if $c\le3.$   
\end{theorem}

\vspace{0.2cm}

Theorem \ref{Mtz-Tau} is due Gerstenhaber \cite[II Theorem 1]{G} and Motzkin and Taussky \cite{MT}. Later Guralnick also gave a short proof of this result in \cite{gural}. One can also see \cite[Theorem 7.6.1]{OCV} for the proof.

Theorem \ref{Meara} is a combination of various papers concerning irreducibility of $\mathcal{C}(3,c).$ In \cite{HO}, Holbrook and Omaldi\v{c} prove that for $c<6,$ $\mathcal{C}(3,c)$ is irreducible while for $c>29$ it is reducible. Then Omaldi\v{c} proved in \cite{Omladic} that $\mathcal{C}(3,c)$ is irreducible for $c=6$, Han proved in \cite{Yongho} that $\mathcal{C}(3,7)$ is irreducible and \v{S}ivic in \cite{S2} showed that $\mathcal{C}(3,8)$ is irreducible. Finally, in \cite{S}, \v{S}ivic also proved that $\mathcal{C}(3,c),$ for $c=9,10.$ In \cite[Proposition 7.9.3]{OCV} one can find a proof for $c\geq29$. 

Theorem \ref{Gersten} is the combination of \cite[Proposition 7.6.8]{OCV}, for the ``if" part and \cite[Corollary 7.6.7]{OCV}, for the ``only if" part.

\bigskip

Now define the affine space $\mathbb{B}:=\End(V)^{\oplus n}\oplus V^{\oplus r}$ whose points are represented by the $(n+r)$-tuple $X=(B_{0},B_{1},\ldots,B_{n-1},v_{1},\cdots,v_{r})$ that will be called an \emph{ ADHM datum}. We then define the \emph{variety of ADHM data} $\mathcal{U}(n,c,r)$ as the subvariety of $\mathbb{B}$ given by
$$ \mathcal{U}(n,c,r):= \mathcal{C}(n,c)\times V^{\oplus r}. $$

\begin{definition}\label{stability} 
An ADHM datum $X=(B_{0},B_{1},\ldots,B_{n-1},v_{1},\cdots,v_{r})\in\mathbb{B}$ is said to be \emph{stable} if there is no proper subspace $S \subsetneq V$ such that 
$$ B_{0}(S),B_{1}(S),\cdots,B_{n-1}(S)\subseteq S,$$ and $S$ contains all the vectors $v_{1},\cdots,v_{r}.$
\end{definition}

The set of stable points in $\mathbb{B}$ will be denoted by $\mathbb{B}^{st}$; $\mathcal{U}(n,c,r)^{st}:= \mathbb{B}^{st}\cap\mathcal{U}(n,c,r)$ will denote the set of stable points in $\mathcal{U}(n,c,r).$

\bigskip

Next, we introduce the action of the linear group $G:=GL(V)$ on $\mathbb{B}$. For all $g\in G$ and $X=(B_{0},B_{1},\ldots,B_{n-1},v_{1},\cdots,v_{r})\in \mathbb{B},$ this action is given by
$$ g\cdot(B_{0},B_{1},\ldots,B_{n-1},v_{1},\cdots,v_{r}) = (gB_{0}g^{-1},\ldots,gB_{n-1}g^{-1},gv_{1},\cdots,gv_{r}). $$
For a fixed ADHM datum $X,$ we will denote by $G_X$ its stabilizer subgroup:
$$ G_X:=\{g \in G \,|\, gX = X\}\subseteq G. $$
It is easy to see that $X$ is stable if and only if $gX$ is stable, and that $G$ acts on $\mathcal{U}(n,c,r)$.

We conclude this section with two results relating stability in the sense of Definition \ref{stability} with GIT stability, following the construction in \cite[\S2]{K}.

\begin{proposition}\label{prop4}
If $X\in\mathcal{U}(n,c,r)^{st}$, then its stabilizer subgroup $G_{X}$ is trivial.
\end{proposition}

\begin{proof}
Let $X=(B_{0},\ldots,B_{n-1},v_{1},\cdots,v_{r})$ be a stable ADHM datum and suppose that there exists an element $g\neq\mathbf 1$ in $G$ such that $gv_{j}=v_{j},$ for $j\in\{1,\cdots,r\}$  and $gB_{i}g^{-1}=B_{i}$ for all $i\in\{0, \ldots,n-1\}.$ Then  $\ker(g-\mathbf1)$ is $B_{i}$-invariant, for all $i\in\{0, \ldots,n-1\},$ and contains all $v_{i}'$s. Since $X$ is stable, then $\ker(g-\mathbf 1)\subset V$ must be equal to $V.$ Hence $g$ must be the identity.
\end{proof}

Let $\Gamma(\mathcal{U}(n,c,r))$ be the ring of regular functions on $\mathcal{U}(n,c,r).$ Fix $l>0,$ and consider the group homomorphism $\chi:G\to\mathbb{C}^{\ast}$ given by $\chi(g)=(\det g)^{l}.$ This can be used for the of construction a suitable linearization of the $G$-action on $\mathcal{U}(n,c,r),$ that is, to lift the action of
$G$ on $\mathcal{U}(n,c,r)$ to an action on $\mathcal{U}(n,c,r)\times\mathbb{C}$ as follows:
$g\cdot(X,z):=(g\cdot X,\chi(g)^{-1}z)$ for any ADHM datum $X\in\mathcal{U}(n,c,r)$ and $z\in\mathbb{C}.$ Then one can form the scheme
$$\mathcal{U}(n,c,r) /\!/_{\chi}G :=
{\rm Proj}\left(\bigoplus_{i\geq0}\Gamma(\mathcal{U}(n,c,r))^{G,\chi^{i}} \right) $$
where
$$\Gamma(\mathcal{U}(n,c,r))^{G,\chi^{i}}:=\left\{ f\in\Gamma(\mathcal{U}(n,c,r)) ~|~
f(g\cdot X)=\chi(g)^{-1}\cdot f(X),\hspace{0.2cm}\forall g\in G\right\}. $$
The scheme $\mathcal{U}(n,c,r)/\!/_{\chi}G$ is projective over the ring $\Gamma(\mathcal{U}(n,c,r))^{G}$ and quasi-projective over
$\mathbb{C}.$

\begin{proposition}\label{closed orbit}
The orbit $G\cdot(X,z)$ is closed, for $z\neq0$, if and only if  the
ADHM datum $X\in\mathcal{U}(n,c,r)$ is stable.
\end{proposition}

\begin{proof}
The proof is similar to \cite[Proposition 2.10]{HJM}.
\end{proof}

From Propositions \ref{prop4}, \ref{closed orbit} and  of \cite[Theorem 1.10]{mumford}, we conclude that the quotient space $\mathcal{M}(n,c,r):=\mathcal{U}(n,c,r) /\!/_{\chi}G$ is a good categorical quotient. Moreover, we conclude from Proposition \ref{closed orbit}, that 
$$ \mathcal{M}(n,c,r) = \mathcal{U}(n,c,r)^{\rm st} / G ,$$ since the GIT quotient $\mathcal{M}(n,c,r)$ is the space of orbits $G\cdot X\subset \mathcal{U}(n,c,r)$ such that the lifted
orbit $G\cdot(X,z)$ is closed within $\mathcal{U}(n,c,r)\times\mathbb{C}$ for all $z\neq0.$


\vspace{1cm}

In this work we consider the Quot scheme $\Quot_{\mathbb{C}^{n}}(c,r)$ classifying quotients $[\mathbb{C}[z_{0},\cdots,z_{n-1}]^{\oplus r}\twoheadrightarrow V],$ where $\mathbb{C}[z_{0},\cdots,z_{n-1}]$ is the polynomial ring in $n$ indeterminates and $V$ is an $n-$dimensional complex vector space.

The existence of its schematic structure is a special case of the general result of Grothendieck \cite{Groth1}. The reader may also consult \cite{Nitsure} for more general results and examples.
An explicit construction of the punctual Quot scheme on the affine plane is given by Baranovsky \cite[Appendix A]{Bara}.
\bigskip

We conclude this section by stating the following:

\begin{theorem}\label{Corresp}
There exists a set-theoretical bijection between the quotient space $\mathcal{M}(n,c,r)$ and the Quot scheme $\Quot_{\mathbb{C}^{n}}(c,r).$
\end{theorem}
\begin{proof}
This is achieved by constructing, for every stable ADHM datum
$$ X=(B_{0},\ldots,B_{n-0},v_{1},\cdots,v_{r})\in\mathcal{U}(n,c,r)^{\rm st}, $$
the surjective map:
\begin{center}
$\begin{array}{cccl} \Phi_{X}: & \mathbb{C}\left[z_{0},\ldots,z_{n-1}\right]^{\oplus r} & \longrightarrow & \qquad V \\ & (p_{j}(z_{0},\ldots,z_{n-1}))_{j\in\{1,\cdots,r\}} & \longmapsto &  \sum_{j=1}^{r}p_{j}(B_{0},\ldots,B_{n-1})v_{j} \end{array}.$
\end{center}
Where $v_{j}$ is just the image of the unit element in the $j-$th copy of $\mathbb{C}\left[z_{0},\ldots,z_{n-1}\right],$ i. e., $\Phi_{X}(0,\cdots,1,\cdots, 0)=:v_{j}.$ Moreover, the map $\Phi_{X}$ is clearly surjective, otherwise there is a proper subspace $S,$ of $V,$ which is $B_{i}-$invariant for all $i\in\{0,\cdots,n-1\}$ and containing all of the vectors $v_{1},\cdots,v_{r}.$

\vspace{0.5cm}

On the other hand, given a surjective map $\Phi: \mathbb{C}\left[z_{0},\ldots,z_{n-1}\right]^{\oplus r}\longrightarrow V,$ the natural action of the coordinate $z_{i}$ on the free module $\mathbb{C}\left[z_{0},\ldots,z_{n-1}\right]^{\oplus r}$ induces an action on the vector space $V$ by means of some endomorphism $B_{i}.$ We also define $v_{j}$ by $\Phi(0,\cdots,1,\cdots, 0).$ Thus, we obtained a datum $X_{\Phi}=(B_{0},\ldots,B_{n-1},v_{1},\cdots,v_{r}).$ Furthermore, since $\Phi$ is surjective, then $V$ is generated, as a $\mathbb{C}\left[z_{0},\ldots,z_{n-1}\right]-$module, by vectors of the form $\sum_{\alpha_{j,i}}a_{\alpha_{j,0}\cdot\alpha_{j,1}} B_{0}^{\alpha_{j,0}}\cdot B_{1}^{\alpha_{j,1}}\cdots B_{n-1}^{\alpha_{j,n-1}}\cdot v_{j},$ in other words $$\mathcal{A}:=\stackunder{\oplus}{\scriptstyle j=1,\cdots,r}< B_{0}^{\alpha_{j,0}}\cdot B_{1}^{\alpha_{j,1}}\cdots B_{n-1}^{\alpha_{j,n-1}}\cdot v_{j}\vert \alpha_{j,i}\in\mathbb{N}\cup\{0\}>$$ is the vector space $V.$ 

\begin{proposition}\label{stab-crit}
For every $X=(B_{0},B_{1},\ldots,B_{n-1},v_{1},\cdots,v_{r})$ one has
\begin{itemize}
\item[(i)] $\mathcal{A}\quad\subseteq\quad V.$ 
\item[(ii)] $\mathcal{A}=V$ if, and only if, $X$ is stable.
\end{itemize}
\end{proposition}

\begin{proof}
The first statement is obvious. For the second statement, we assume that $\mathcal{A}$ is a proper subspace of $V.$ Then one has $B_{i}(\mathcal{A})\subseteq \mathcal{A},$ $\forall i=0,\cdots, n-1$ and $\im(v_{i})\subseteq \mathcal{A}.$ Hence the datum is not stable. Finally, if $X$ is not stable. Let $S\subsetneq V$ the stabilizing subspace, i.e, $S$ is $B_{i}-$invariant, for all $i,$ and contains all $v_{j}.$ It follows that $\mathcal{A}\subseteq S\subsetneq V,$ Thus it cannot be $V.$  

\end{proof}

Thus, as one might expect, the datum $X_{\Phi}$ obtained from $\Phi: \mathbb{C}\left[z_{0},\ldots,z_{n-1}\right]^{\oplus r}\longrightarrow V,$ is stable. Showing the bijection is an exercise left to the reader.

\end{proof}


\section{Extended monads and perfect extended monads}\label{ext-monads}

This section is a reminder of some useful results from \cite[Section 3]{HJ}. The proofs will be omitted since they can be found in the above reference.

Let $X$ be a smooth projective algebraic variety of dimension $n$ over the field of complex numbers $\C$, and let $\mathcal{O}_{X}(1)$ be a polarization on it.


\subsection{$l$-extended monads}

\begin{definition}\label{l-ext}
An \emph{$l$-extended monad} over $X$ is a complex
\begin{equation}\label{premordial}
C^{\bullet}:\quad  0\to C^{-l-1}\stackrel{\alpha_{-l-1}}{\longrightarrow} C^{-l}\stackrel{\alpha_{-l}}{\longrightarrow}\cdots\stackrel{\alpha_{-2}}{\longrightarrow} C^{-1}\stackrel{\alpha_{-1}}{\longrightarrow} C^{0}\stackrel{\alpha_{0}}{\longrightarrow} C^{1}\to0
\end{equation}
of locally free sheaves over $X$ which is exact at all but the $0$-th position, i.e. $\mathcal{H}^{i}(C^{\bullet})=\ker\alpha_i/\im\alpha_{i-1}=0$ for $i\ne0$. 
The coherent sheaf $\mathcal{E}:=\mathcal{H}^{0}(C^{\bullet})=\ker \alpha_{0}/\im \alpha_{-1}$ will be called \emph{the cohomology of $C^{\bullet}$}.
 \end{definition}
 
Note that a monad on $X$, in the usual sense, is just a $0$-extended monad.

Moreover, one can associate to any $l$-extended monad $C^{\bullet}$ a \emph{display} of exact sequences as the following
\begin{equation}\label{ext-display}
\xymatrix@R-2pc@C-0.5pc@u{& 0\ar[d] & 0\ar[d] &  & \\
& C^{-l-1}\ar[d]_{\alpha_{-l-1}}\ar@{=}[r] & C^{-l-1}~~~~~~\ar[d]^{\alpha_{-l-1}} & & \\
& C^{-l}\ar@{=}[r]\ar[d]_{\alpha_{-l}} & C^{-l}\ar[d]^{\alpha_{-l}} & & \\
& \vdots\ar[d]_{\alpha_{-2}} & \vdots\ar[d]^{\alpha_{-2}} & & \\
& C^{-1}\ar[d]\ar@{=}[r] & C^{-1}\ar[d]^{\alpha_{-1}} & & \\
0\ar[r]& K\ar[d]\ar[r] & C^{0}\ar[d]\ar[r]^{\alpha_{0}} & C^{1}\ar@{=}[d]\ar[r] &0 \\
0\ar[r]& \mathcal{E}\ar[d]\ar[r] & Q\ar[d]\ar[r] & C^{1}\ar[r] & 0\\
& 0& 0 &  &
}\end{equation}
where $K:=\ker\alpha_{0}$ and $Q:=\coker\alpha_{-1}$.

An $l$-extended monad can be broken into the following two complexes: first,
\begin{equation}\label{Resol}
\xymatrix@C-0.5pc{ N^{\bullet}:&0\ar[r]&C^{-l-1}~~~~~~\ar[r]^{\alpha_{-l-1}}&C^{-l}&\cdots\ar[r]^{\alpha_{-3}}&C^{-2}\ar[r]^{\alpha_{-2}} & C^{-1}\ar[r]^{J_{-1}}&\mathcal{G}\ar[r]&0
} \end{equation}
which is exact, and a locally free resolution of the sheaf $\mathcal{G}=\coker\alpha_{-2},$ and
\begin{equation}\label{Monad-like}
\xymatrix@C-0.5pc{ M^{\bullet}:& \mathcal{G}\ar[r]^{I_{-1}} & C^{0}\ar[r]^{\alpha_{0}}&C^{1}
} \end{equation}
where $I_{-1}\circ J_{-1}=\alpha_{-1}$. $M^{\bullet}$ is a monad-like complex in which the coherent sheaf $\mathcal{G}$ might not be locally free; indeed, $\mathcal{G}$ is not locally free for the extended monads describing ideal sheaves of $0$-dimensional subschemes, the situation most relevant to the present paper. 

For a given $l$-extended monad, we refer to the complexes $M^{\bullet}$ and $N^{\bullet}$ as the \emph{associated resolution} and the \emph{associated monad}, respectively.

\begin{definition}
A \emph{perfect extended monad} on a $n$-dimensional projective \linebreak variety $X$ is a $(n-2)$-extended monad $P^{\bullet}$ on $X$ of the following form
$$\xymatrix@C-0.8pc{ 0\ar[r]&\mathcal{O}_{X}(1-n)^{\oplus a_{1-n}}\ar[r]^{\alpha_{1-n}}&
\mathcal{O}_{X}(2-n)^{\oplus a_{2-n}}\ar[r]&}\hspace{1cm}$$
$$\hspace{3cm}\xymatrix@C-0.5pc{\cdots\ar[r]^{\alpha_{-2}\hspace{0.5cm}}&\mathcal{O}_{X}(-1)^{\oplus a_{-1}}\ar[r]^{\hspace{0.3cm}\alpha_{-1}} &\mathcal{O}_{X}^{\oplus a_{0}}\ar[r]^{\alpha_{0}\hspace{0.1cm}}&\mathcal{O}_{X}(1)^{\oplus a_{1}}\ar[r]&0
},$$ for some integers $a_{i},$ $1-n\leq i\leq 1.$ 
\end{definition}

We denote by $Kom^{\flat}(X)$ the category whose objects are bounded complexes of sheaves over the projective scheme $X$ and whose morphisms are degree zero morphisms of complexes. For standard results about this category one might consult \cite[Chapter ${\rm III}, \S 1.$]{GM}

We recall to the reader that a projective scheme $X$ is \emph{arithmetically Cohen--Macaulay}, or simply \emph{ACM}, if its homogeneous coordinate ring is Cohen--Macaulay ring. Moreover let us denote by $\mathfrak{P}er$ the full subcategory of $Kom^{\flat}(X)$ consisting of perfect extended monads. In this case, the cohomology functor
$$ H:\mathfrak{P}er(X)\to{\rm Coh}(X) $$
is full and faithful\cite[Corollary 3.5]{HJ}. For projective spaces, one can characterizes sheaves which are in the image of the functor $H$ $n\geq 2$ as in the following.

\begin{proposition}\label{hom-char}
If $\mathcal{E}$ is the cohomology of a perfect extended monad on $\pn$ ($n\ge 2$) then:
\begin{itemize}
\item[(i)] $\ho^{0}(\mathcal{E}(k))=0$ for $k<0$;
\item[(ii)] $\ho^{n}(\mathcal{E}(k))=0$ for $k>-n-1$;
\item[(iii)] $\ho^{i}(\mathcal{E}(k))=0$ $\forall k,$ $2\leq i\leq n-1$, when $n\ge3$.
\end{itemize}
\end{proposition}

Now let us denote by $\Omega^{-p}_{\pn}$ the bundle of holomorphic $(-p)$-forms on $\pn,$ where $p\leq0$ in our convention.

\begin{proposition}\label{Character}
If a coherent sheaf $\mathcal{E}$ on $\pn$ ($n\ge 2$) satisfies:
\begin{itemize}
\item[(i)] $\ho^{0}(\mathcal{E}(-1))=\ho^{n}(\pn,\mathcal{E}(-n))=0$;
\item[(ii)] $\ho^{q}(\mathcal{E}(k))=0 \quad\forall k,\quad2\leq q\leq n-1$ when $n\ge3$;
\item[(iii)] $\ho^{1}(\mathcal{E}\otimes\Omega^{-p}_{\pn}(-p-1))\neq0$ for $-n\le p\leq0$;
\end{itemize}
then $\mathcal{E}$ is the cohomology of a perfect extended monad.
\end{proposition}

\vspace{0.3cm}


\section{Quotients of $\mathcal{O}_{\pn}^{\oplus r}$ supported on a zero-dimensional subscheme of $\pn$}\label{z-cy}

\vspace{0,5cm}

We now consider sheaves $\mathcal{E}$ of rank $r$ on $\pn$ fitting in the following short exact sequence
\begin{equation}\label{structure}
0 \to \mathcal{E} \to \mathcal{O}_{\pn}^{\oplus r} \to \mathcal{Q} \to 0,
\end{equation}
where $\mathcal{Q}$ is a pure torsion sheaf of length $c$ supported on a $0$-dimensional subscheme $Z\subset\pn$. 

Note that the Chern character of $\mathcal{E}$ is given by $ch(\mathcal{E})=r-cH^{n}$, and that $\mathcal{E}$ is necessarily torsion free. Such sheaves can also be regarded as points in the Quot scheme $Quot^{P=c}(\mathcal{O}_{\pn}^{\oplus r})$. 

One can easily see that $\mathcal{E}$ is $\mu-$semi-stable; A subsheaf $\mathcal{F},$ of $\mathcal{E},$ should satisfy $c_{1}(\mathcal{F})\leq0$ since it is also a subsheaf of $\mathcal{O}_{\mathbb{P}^{n}}^{\oplus r}.$

In the case $r=1$, it is clear that $\mathcal{E}$ is the sheaf of ideals in $\mathcal{O}_{\pn}$ associated to the zero-dimensional subscheme $Z.$ 

\begin{proposition}\label{perfect-quot}
Every sheaf $\mathcal{E}$ on $\pn$ given by sequence \eqref{structure} is the cohomology of a perfect extended monad $P^{\bullet}$ with terms of the form $P^{-i}:=V_{i}\otimes\mathcal{O}_{\pn}(i)$, $i=1-n,\dots,0,1$, where
\begin{equation}\label{eq-lema}
V_{i}:=\ho^{1}(\mathcal{E}\otimes\Omega^{1-i}_{\pn}(-i)) \cong
\ho^{0}(\mathcal{Q}\otimes\Omega^{1-i}_{\pn}(-i)).
\end{equation}
Furthermore, we have the following isomorphisms:
\begin{equation}\label{id1}
V_1 \cong \ho^{0}(\mathcal{Q})
\end{equation}
\begin{equation}\label{id2}
V_{i}\cong  
\left\{\begin{array}{ll} V_{1}^{\oplus n}\oplus\mathbb{C}^{r} & \textnormal{for }i=0 \\ 
V_{1}^{\oplus \binom{n}{1-i}}  &  \textnormal{for }i<0 \end{array}\right.
\end{equation}
\end{proposition}


One can repeat the procedure used in  \cite[Section 4]{HJ} to obtain the following reduced form of the perfect extended monad maps

\begin{equation}\label{monadmaps} \begin{split}
&\alpha_{0}=\left(\begin{array}{lllllll} B_{0}z_{n}-z_{0}&B_{1}z_{n}-z_{1} & \cdots&B_{n-1}z_{n}-z_{n-1},&v_{1}z_{n}&\cdots&v_{r}z_{n}\end{array}\right). \\
&\alpha_{-1}=\left(\begin{array}{llll} A_{0} & A_{1}  & \cdots & A_{n-2} \\ \hspace{0.1cm}0 & \hspace{0.1cm}0 & \cdots & \hspace{0.3cm}0 \\ \hspace{0.1cm}\vdots & \hspace{0.1cm}\vdots & \cdots & \hspace{0.3cm}\vdots \\ \hspace{0.1cm}0 & \hspace{0.1cm}0 & \cdots & \hspace{0.3cm}0 \end{array}\right).
\end{split}
\end{equation}
where each block $A_{i},$ $0\leq i\leq n-2$ is an $[ n\cdot c \times (n-i-1)\cdot c]$-matrix of the form
{\small
\begin{displaymath}
A_{i}=\left( \begin{array}{lllll}
\hspace{1cm}0&\hspace{1cm}0&\hspace{1cm}0&\cdots&\hspace{1cm}0\\
\hspace{1cm}0 &\hspace{1cm}0&\hspace{1cm}0&\cdots& \hspace{1cm}0\\
\hspace{1cm}\vdots &\hspace{1cm}\vdots&\hspace{1cm}\vdots&\cdots&\hspace{1cm}\vdots \\
\hspace{1cm}0&\hspace{1cm}0&\hspace{1cm}0&\cdots&\hspace{1cm}0\\
B_{i+1}z_{n}-z_{i+1}&B_{i+2}z_{n}-z_{i+2}&B_{i+3}z_{n}-z_{i+3}&\cdots&B_{n-1}z_{n}-z_{n-1} \\
-B_{i}z_{n}+z_{i} &\hspace{1cm}0&\hspace{1cm}0&\cdots&\hspace{1cm}0 \\
 \hspace{1cm}0&-B_{i}z_{n}+z_{i}&\hspace{1cm}0&\cdots&\hspace{1cm}0 \\
 \hspace{1cm}0&\hspace{1cm}0&-B_{i}z_{n}+z_{i}&\cdots&\hspace{1cm}0 \\
 \hspace{1cm}0&\hspace{1cm}0&\hspace{1cm}0&\cdots&\hspace{1cm}0 \\
\hspace{1cm}\vdots&\hspace{1cm}\vdots&\hspace{1cm}\vdots&\ddots&\hspace{1cm}0 \\
 \hspace{1cm}0&\hspace{1cm}0&\hspace{1cm}0&\cdots&-B_{i}z_{n}+z_{i}
\end{array}\right)
\end{displaymath}
}
The first non vanishing line in $A_{i}$ is $(i+1)$-est one. For instance, when $n=3$ there are two blocks $$A_{0}=\left(\begin{array}{ll} \hspace{0.3cm}B_{1}z_{3}-z_{1} & \hspace{0.3cm}B_{2}z_{3}-z_{2} \\ -B_{0}z_{3}+z_{0}&\hspace{0.8cm}0\\ \ \hspace{0.8cm}0&-B_{0}z_{3}+z_{0} \end{array}\right)
\hspace{1cm} A_{1}=\left(\begin{array}{l}\hspace{0.8cm}0\\ \hspace{0.3cm}B_{2}z_{3}-z_{2}\\-B_{1}z_{3}+z_{1} \end{array}\right),$$ of respective sizes $[3c \times 2 c]$ and $[3 c \times  c].$
The reader may check that $\alpha_{0}\circ\alpha_{-1}=0\Leftrightarrow [B_{i},B_{j}]=0,$ for all $0\leq i,j\leq n-1,$ and that the following result holds, as in \cite[Lemma 4.4 \& Theorem 4.5]{HJ}.
 
\vspace{0.3cm}

\begin{theorem}[Inverse construction]\label{reconstruction}
\hspace{7cm}
\begin{itemize}
\item[(i)] The map $\alpha_{0}$ given above is surjective if and only if the ADHM datum $(B_{0},\cdots,B_{n-1},v_{1},\cdots,v_{r})$ is stable.
\item[(ii)] To a stable ADHM datum $X=(B_{0},\cdots,B_{n-1},v_{1},\cdots,v_{r})\in\mathcal{U}(n,c,r)^{st}$ one can associate the perfect extended monad as in \eqref{perfect-quot} with maps $\alpha_{-2},\alpha_{-1},\alpha_{0}$ given as in \eqref{monadmaps}, such that its cohomology is an ideal sheaf whose restriction to $\mathbb{C}^{n}=\mathbb{P}^{n}\backslash\wp$ is isomorphic to the one given by Theorem \ref{Corresp}.
\end{itemize}
\end{theorem}

We end this section by mentioning the following results about the representability of the Quot functor:

\begin{theorem}\label{represent}
The scheme $\mathcal{M}(n,c,r)$ is a fine moduli space for the Quot functor
$\mathcal{Q}\textnormal{uot}_{\C^n}(c,r){\bf(\cdot)}$ on $\mathbb{C}^{n}$.
\end{theorem}

\begin{proof}
The proof is similar, \emph{mutatis mutandis}, to that of \cite[Theorem 4.2]{HJM}.
\end{proof}

A similar result also holds for affine varieties, in general. Following the notations in \cite[Section 5]{HJ}, we let $\mathbb{Y}=\mathcal{Z}(Z_{\mathbb{Y}})\subset\mathbb{C}^{n}$ be an affine variety, given by the zero locus of the ideal $Z_{\mathbb{Y}}\subsetneq\mathbb{C}[x_{1},\cdots,x_{n}].$ We denote by $A(\mathbb{Y})$ the affine coordinate ring of the variety $\mathbb{Y},$ i.e., $A(\mathbb{Y})=\frac{\mathbb{C}[x_{1},\cdots,x_{n}]}{Z_{\mathbb{Y}}}.$ 

Defining the set
$$ \mathcal{U}_{\mathbb{Y}}(c,r)^{st} := \{(B_{0},\cdots,B_{n-1},v_{1},\cdots,v_{r})\in\mathcal{U}(n,c,r)^{st}/\begin{array}{l}f(B_{0},\cdots,B_{n-1})=0,\\ \forall f\in Z_{\mathbb{Y}} \end{array}\}. $$ 
we have the following generalization of \cite[Theorem 5.1]{HJ}:
\begin{theorem}\label{correspond2}
The scheme $\mathcal{M}_{\mathbb{Y}}(c,r) := \mathcal{U}_{\mathbb{Y}}(c,r)^{st} / GL(V)$ is a fine moduli space for the Quot functor $\mathcal{Q}\textnormal{uot}_{\mathbb{Y}}(c,r)(\cdot)$ on $\mathbb{Y}.$
\end{theorem}

\subsection{The $\p3$ case}\label{reduction}

Now, we fix a hyperplane $\wp\subset\p3$. We shall describe how to get linear algebraic data out of the perfect extended monad corresponding to a quotient  $\mathcal{O}_{\pn}^{\oplus r} \twoheadrightarrow \mathcal{Q}$ where the purely torsion sheaf $\mathcal{Q}$ is supported on a $0$-dimensional
subscheme. 

 We first choose homogeneous coordinates $[z_{0};z_{1};z_{2};z_{3}]$ on $\p3$
in such a way that the hyperplane $\wp$ is given by the equation $z_{3}=0$. We also regard such coordinates as a basis for the space of global sections $\ho^{0}(\mathcal{O}_{\p3}(1))$. 

By Proposition \ref{perfect-quot}, there is a perfect extended monad $P^{\bullet}$ with cohomology equal to the ideal sheaf $\mathcal{E}$. It is given by
\begin{equation}\label{cpx-p3}
{\small \xymatrix@C-1pc{
0\ar[r]&V_{1}\otimes\mathcal{O}_{\p3}(-2)\ar[r]^{\alpha_{-2}} & V_{1}^{\oplus3}\otimes\mathcal{O}_{\p3}(-1)\ar[r]^{\alpha_{-1}\hspace{0.5cm}}& (V_{1}^{\oplus3}\oplus\mathbb{C}^{r})\otimes\mathcal{O}_{\p3}\ar[r]^{\alpha_{0}\hspace{0.2cm}} & V_{1}\otimes\mathcal{O}_{\p3}(1)\ar[r]&0
}}
\end{equation}

\noindent where the maps  maps $\alpha_{-2},$ $\alpha_{-1}$ and $\alpha_{0}$ are given by:
\begin{equation}\label{alpha3}
\begin{split}
\alpha_{-2}=&\left(\begin{array}{l} -B_{2}z_{3}+z_{2} \\\hspace{0.3cm} B_{1}z_{3}-z_{1}\\ -B_{0}z_{3}+z_{0}\end{array}\right);
\quad\alpha_{-1}=\left(\begin{array}{lll} \hspace{0.3cm}B_{1}z_{3}-z_{1} & \hspace{0.3cm}B_{2}z_{3}-z_{2} &\hspace{0.8cm}0\\ -B_{0}z_{3}+z_{0}&\hspace{0.8cm}0& \hspace{0.3cm}B_{2}z_{3}-z_{2}\\ \hspace{0.8cm}0&-B_{0}z_{3}+z_{0}&-B_{1}z_{3}+z_{1} \\ \hspace{0.8cm}0&\hspace{0.8cm}0&\hspace{0.8cm}0\\ \hspace{0.8cm}\vdots & \hspace{0.8cm}\vdots & \hspace{0.8cm}\vdots \\ \hspace{0.8cm}0 & \hspace{0.8cm}0 & \hspace{0.8cm}0\end{array}\right); \\ 
&\alpha_{0}=\left(\begin{array}{llllll} B_{0}z_{3}-z_{0}&B_{1}z_{3}-z_{1} & B_{2}z_{3}-z_{2}&v_{1}z_{n}&\cdots&v_{r}z_{n}\end{array}\right).
\end{split}
\end{equation}
such that
\begin{equation}
[B_{0},B_{1}]=0;\quad[B_{0},B_{2}]=0;\quad[B_{1},B_{2}]=0,
\end{equation}
Moreover the ADHM datum
$ (B_{0},B_{1},B_{2},v_{1},\cdots,v_{r})\in\End(V_{1})^{\oplus3}\oplus V_{1}^{\oplus r} $ is indeed stable. 

\vspace{0.3cm}

\section{Relation with Quivers and their representations}\label{quivers}

Another way of reorganizing the linear algebraic data is through the notion of \emph{quivers} \cite{K,BDJ,FJM,JP,N1,N2}. A quiver $\mathcal{Q}$ is the data $(\mathcal{Q}_{0}, \mathcal{Q}_{1}, s,t:\mathcal{Q}_{1}\longrightarrow\mathcal{Q}_{0})$ consisting of a finite set $\mathcal{Q}_{0}$ of vertices, a finite set $\mathcal{Q}_{1}$ of arrows between the vertices and maps, $s$ and $t,$ assigning to every arrow $\rho,$ the vertex $t(\rho),$ called its target, and the vertex $s(\rho),$ called its source.

A \emph{$\mathcal{Q}-$ representation} $R$ of a quiver $\mathcal{Q}$ is given by $R=\{(V_{\nu},b_{\rho})_{\nu\in\mathcal{Q}_{0},\rho\in \mathcal{Q}_{1}} | b_{\rho}:V_{s(\rho)}\to V_{t(\rho)}\},$ where $V_{\nu}$ is a $\mathbb{C}-$vector space and $b_{\rho}$ is a $\mathbb{C}-$linear map. 
For each $\mathcal{Q}-$representation, one can associate the vector $\vec{m}:=({\rm dim}V_{\nu})_{v\in\mathcal{Q}_{0}}$ called the \emph{dimension vector}. 
A morphism of $\mathcal{Q}-$representations $R=(V_{\nu},b_{\rho})$ and $S=(S_{\nu},a_{\rho})$ is collection of $\mathbb{C}-$ linear maps $L_{\nu}:V_{\nu}\to S_{\nu},$ for all $\nu\in\mathcal{Q}_{0},$ making the following diagram commute:
\begin{equation}
   \xymatrix@R-1pc@C-1pc{
   V_{s(\rho)}\ar[r]^{b_{\rho}}\ar[d]_{L_{s(\rho)}}& V_{t(\rho)}\ar[d]^{L_{t(\rho)}} \\
   S_{s(\rho)}\ar[r]^{a_{\rho}}& S_{t(\rho)}} 
\end{equation}

for all $\rho\in\mathcal{Q}_{1}$.

A path of length $l$ of a quiver $\mathbb{Q}$ is a composition of $l$ arrows in $\mathcal{Q}_{1},$ i.e., $\rho_{1}\rho_{2}\cdots\rho_{l}$ with $\rho_{1}\in\mathcal{Q}_{1}$ satisfying $t(\rho_{i})=s(\rho_{i+1}).$ The path algebra of a quiver $\mathcal{Q}$ is a $\mathbb{C}-$vector space spanned by paths in $\mathcal{Q}:$ 

$$\mathbb{C}[Q]:=\bigoplus_{l\geq0}\quad\bigoplus_{\begin{subarray}{l} \rho_{1}\rho_{2}\cdots\rho_{l} \\ t(\rho_{i})=s(\rho_{i+1})
      \end{subarray}}\mathbb{C}\cdot\rho_{1}\rho_{2}\cdots\rho_{l}.$$

A \emph{relation} on a quiver is a formal $\mathbb{C}-$linear combination of paths in $\mathbb{C}[Q],$ and a quiver with relations is a pair $(\mathcal{Q},I)$ consisting of a quiver and an ideal $I\subset\mathbb{C}[Q],$ of relations.

As in \cite[\S 3. p.870]{BDJ} one  can define a \emph{framed quiver with relations} as the triple $(\mathcal{Q},I,h),$ where $(\mathcal{Q},I)$ is a quiver with relations and $h:V_{\nu_{0}}\to\mathbb{C}^{r}$ is a fixed isomorphism, for a fixed vertex $\nu_{0}\in\mathcal{Q}_{0}.$

In view of theorems \ref{represent} and \ref{correspond2} we will reinterpret the Quot scheme ${\rm Quot}_{\mathbb{Y}}(c,r)$ as the the moduli space of $\mathbb{C}-$linear representations of the following framed quiver with relations:

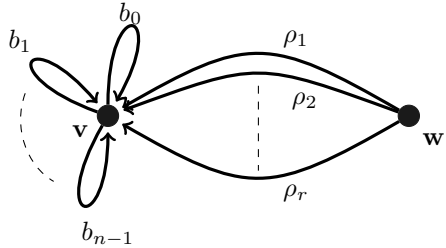
\begin{center}
\begin{tikzpicture}
[inner sep=1mm, shorten >=2pt, vertex/.style={circle, draw=black!90,fill=black!90}, frame/.style={circle, draw=black!90,fill=black!90}]
\node (virt1)  at (0.05,0.05) [] {};
\node (virt2)  at (0,0.05) [] {};
\node (virt3)  at (0,-0.05) [] {};

\node (vert) at (0,0) [vertex, label=185:$\bf{v}$] {}
	edge [loop, out=90, in=60, very thick, distance=1.5cm,->] (vert) {}
	edge [loop, out=240, in=270, very thick, distance=1.5cm,->] (vert) {}
	edge [loop, out=160, in=130, very thick, distance=1.5cm,->] (vert) {};

\node (fra) at (4,0) [frame,  label=300:$\bf{w}$] {}
	edge [out=120, in=60, very thick, bend right ,distance=2cm, ->] (vert) {}
	edge [out=160, in=20, very thick, distance=2cm, ->] (vert) {}
	edge [out=240, in=330, very thick, bend left ,distance=2cm, ->] (vert) {};

\node (b0) at (0.3,1) [label=above:$b_{0}$] {};
\node (b1) at (-0.8,1) [label=left:$b_{1}$] {};
\node (bn-1) at (0,-1.2) [label=below:$b_{n-1}$] {};

\node (r0) at (2.5,0.7) [label=above:$\rho_{1}$] {};
\node (r1) at (3,0.2) [label=left:$\rho_{2}$] {};
\node (rr) at (2.5,-0.7) [label=below:$\rho_{r}$] {};

\draw[dashed]  (-1.1,0.3) arc (160:240:1);

\draw[dashed] (2,0.4) -- (2,-0.8); 

\end{tikzpicture}
\end{center}
that we will call the \emph{extended ADHM quiver}, and denote by $\mathcal{Q}_{\mathbb{Y}}.$ The framing is given by $h:\bf{w}\to\mathbb{C}$ and the ideal of relations, over $\mathbb{Y},$ is given by $<f(B_{0},\cdots,B_{n-1}), \lbrack B_{i},B_{j}\rbrack>,$ for $i=0,\cdots,n-1$ and $f\in Z_{\mathbb{Y}},$ as in the proof of Theorem \ref{represent}. We will introduce a notion of stability, in order to define the moduli space of representations of the above quiver. As in \cite{K} we define a stability condition with respect to the parameter $\vartheta=(\theta, \theta_{\infty})$ satisfying $c\cdot\theta + \theta_{\infty}=0.$

\begin{definition}\label{theta:stab}
A representation $R,$ with dimension vector $(c,1),$ of the extended ADHM quiver $\mathcal{Q}_{\mathbb{Y}}$ is said to be $\vartheta$(semi-)stable  if for any subrepresentation, $S$ of $R,$ the following conditions are satisfied
\begin{itemize}
    \item[(i)] if $S$ has dimension vector $(c',0),$ then $c'\cdot\theta\hspace{0.2cm} (\leq) <0,$ and 
    \item[(ii)] if $S$ has dimension vector $(c',1),$ then $c'\cdot\theta + \theta_{\infty} \hspace{0.2cm} (\leq) <0.$
\end{itemize}
\end{definition}

We remark that the above definition, coincides with the definition of King's stability \cite[\S 3]{K}.

\begin{lemma}\label{theta:stab=stab}
Suppose $\theta <0.$ Then a representation $R,$ of the extended ADHM quiver $\mathcal{Q}_{\mathbb{Y}},$ with dimension vector $(c,1)$ is $\vartheta-$stable if and only if $R$ is $\vartheta-$semi-stable and if the stability condition \ref{stability} is satisfied.
\end{lemma}

\begin{proof}
    Suppose $R$ is $\vartheta-$semi-stable and assume that there is a subspace $S\subsetneq V$ such that $B_{i}(S)\subseteq S$ and $\im(v_{j})\subset S,$ $\forall i=0,\cdots,n-1,$ $\forall j=1,\cdots,r.$ Then $S$ has dimension vector $(c',1)$ with $0\leq c'<c.$ It follows that $c'\cdot\theta + \theta_{\infty}=(c'-c)\cdot\theta>0,$ since $\theta_{\infty}=c\cdot\theta.$ But this is absurd. Hence stability condition \ref{stability} must be satisfied.    
    On the other hand, suppose that stability condition \ref{stability} is satisfied. It follows that every subrepresentation has a dimension vector $(c',0).$ Thus $c\dot\theta<0,$ and it follows that $R$ is stable.
\end{proof}

This lemma insures that there exists a notion of generic stability parameters for any dimension vector $(c,1).$ One can form the quotient $\mathcal{R}er^{st}_{\mathcal{Q}_{\mathbb{Y}}}(c,1)//_{\vartheta}GL(V),$ of the space of $\vartheta-$stable representations of the extended ADHM quiver, with fixed dimension vector $(c,1),$ with respect to the action of $GL(V).$ Then by Lemma \ref{theta:stab=stab}, Theorems \ref{correspond2}, \ref{represent} and universality of the Quot scheme one has $$\mathcal{R}er^{st}_{\mathcal{Q}_{\mathbb{Y}}}(c,1)//_{\vartheta}\hspace{0.2cm} GL(V)=\Quot_{\mathbb{Y}}(c,r).$$


\section{Irreducible components of the Quot scheme of points}\label{irred}

Not much is known about the irreducible components of the Quot scheme
$\Quot_{\mathbb{C}^{n}}(c,r),$ in general, except for $n=2,$ \cite{Elling} $\&$ \cite[Appendix. A]{Bara}.  In this section, we answer the question in the higher dimensional case, at least in some cases for small values of $c.$ 

\begin{proposition}\label{number}
The number of irreducible components of $\Quot_{\mathbb{C}^{n}}(c,r)$ is smaller than, or equal to, the number of irreducible components of $\mathcal{C}(n,c)$. In particular, if $\mathcal{C}(n,c)$ is irreducible, then $\Quot_{\mathbb{C}^{n}}(c,r)$ is also irreducible.
\end{proposition}

\begin{proof}
As in \cite[Proposition 6.2]{HJ}, the number of irreducible components of $\mathcal{C}(n,c)$ is the same as the number of irreducible components of $\mathcal{U}(n,c,r):=\mathcal{C}(n,c)\times V^{\oplus r}$. Let
$\mathcal{U}_{1}(n,c,r),\dots,$ $\mathcal{U}_{p}(n,c,r)$ denote the irreducible components of
$\mathcal{U}(n,c,r)$, and set $\mathcal{U}_{l}(n,c,r)^{\rm st}:=\mathcal{U}_{l}(n,c,r)\cap\mathcal{U}(n,c,r)^{\rm st}$, with $l=1,\dots,p$.

Some components of $\mathcal{U}(n,c,r),$ possibly, do not contain any stable points; then up to reordering the irreducible components of $\mathcal{U}(n,c,r)$ we have 
$\mathcal{U}_l(n,c,r)^{\rm st}\ne\emptyset$ for $l=1,\dots,q$ and $\mathcal{U}_l(n,c,r)^{\rm st}=\emptyset$ for $l=q+1,\dots,p$.

Moreover, by the irreducibility of $G:=GL(V),$ when $x\in\mathcal{U}_{l}(n,c,r)^{\rm st},$ one has $G\cdot x \subset \mathcal{U}_{l}(n,c,r)^{\rm st}.$ Note also that
$$ \mathcal{U}_{l}(n,c,r) /\!/_{\chi} G = \mathcal{U}_{l}(n,c,r)^{\rm st}/G $$
is irreducible, for each $l=1,\dots,q$. 

Since the GIT quotient $\mathcal{M}(n,c,r)$ coincides, by Proposition \ref{closed orbit}, with the set of stable $G$-orbits, we have that 
$$ \mathcal{M}(n,c,r) = 
\left( \mathcal{U}_1(n,c,r)^{\rm st}/G \right) \cup \cdots \cup 
\left( \mathcal{U}_q(n,c,r)^{\rm st}/G \right) $$
and the desired conclusion follows from Theorem \ref{represent}.
\end{proof}

\begin{corollary} \label{small-irred}
\hspace{6cm}

\begin{itemize}
\item[(i)] $\Quot_{\mathbb{C}^{3}}(c,r)$ is irreducible for $c\le 10,$ for any $r\geq 1.$

\item[(ii)] $\Quot_{\mathbb{C}^{n}}(c,r)$ is irreducible for $c\leq3$ and $\forall n\geq4,$ and $r>0.$ Moreover, if $c=2$ it is also reduced for every $n,r\geq 1.$

\end{itemize}
\end{corollary}

\begin{proof}
Combining Proposition \ref{number} with Theorem \ref{Meara} and Theorem \ref{Gersten} we obtain the results in items 
\emph{(i)}, and the irreducibility part of \emph{(ii)}, respectively.

For reducedness, we have that $\mathcal{C}(n,2)$ is irreducible and normal by \cite[Theorem 5.2.8]{Ngo}. Then the open dense $\mathcal{U}(n,2,r)^{st}\subset\mathcal{C}(n,2)\times V^{\oplus r}$ is also irreducible and normal. Remark that a normal scheme is reduced. Hence by \cite[p. 5]{mumford} the quotient of $\mathcal{U}(n,2,r)^{st}$ by $G$ is also reduced. 

\end{proof}

\subsection{The punctual Quot scheme}

In the following we will denote by $\Quot_{\mathbb{Y}}^{[p]}(c,r),$ the punctual Quot scheme of points on a given smooth affine variety $\mathbb{Y},$ of dimension $n.$ In this case, the quotient sheaf $\mathcal{Q},$ in \eqref{structure}, is supported on a single point $p.$ Since the matter is local, we can clearly set $\mathbb{Y}=\mathbb{C}^{n}$ and choose the topological support $p$ to be the origin $O=(0,\cdots,0)\in\mathbb{C}^{n}.$ Moreover, the coordinates of the topological support are given by the spectra of the matrices $B_{i},$ $i=0,\cdots,n-1.$ Hence $B_{i}$ should be nilpotent for all $i=0,\cdots,n-1.$ As usual, the nilpotent case is more difficult then the general one, especially in studying the irreducible components. However using the same argument of Proposition \ref{number}; one has the following:

\begin{corollary}\label{quot:pt}
\hspace{6cm}

\begin{itemize}
\item[(i)] $\Quot_{\mathbb{Y}}^{[p]}(c,r)$ is irreducible for $c\le 6,$ for any $r\geq 1$ and ${\rm dim}\mathbb{Y}=3,$

\item[(ii)] $\Quot_{\mathbb{Y}}^{[p]}(c,r)$ is irreducible for $c\leq 3 ,$ for any $n,r\geq 1.$ Moreover, for $c=2$ it is reduced of dimension $2r+n-3$ and for $c=3$ its dimension is $2n+3r-5.$
\end{itemize}

\end{corollary}

\begin{proof}

\begin{itemize}

\item[(i)] This result is obtained from Proposition \ref{number} by replacing $\mathcal{U}(n,c,r)$ with
\begin{align}
\mathcal{N}(n,c,r):=&\{X\in\mathcal{U}(n,c,r)| B_{i}\textnormal{ is nilpotent,} \forall i\in\{0,\cdots, n-1\} \} \notag \\
=&\mathcal{N}(n,c)\times V^{\oplus r},  \notag
\end{align}
where $\mathcal{N}(n,c)$ is the variety of $n$ commuting nilpotent matrices, i.e., $n-$tuples $(B_{0},\cdots,B_{n-1})$ of $c\times c$ matrices such that $B_{i}$ is nilpotent and $\lbrack B_{i},B_{j}\rbrack=0$ for all $0\leq i\leq j\leq n-1.$ Theorem \ref{correspond2} insures that the GIT quotient $\mathcal{N}(n,c,r)^{st}/G$ is indeed $\Quot_{\mathbb{C}^{3}}^{[p]}(c,r).$  The result follows from the irreducibility of $\mathcal{N}(3,c)$ for $c\leq6$ as proved in \cite[Theorem 24]{Ngo-Sivic}.
\vspace{0.5cm}

\item[(ii)] Remark that $\mathcal{N}(n,c)$ can also be seen as the nilpotent commuting variety over $\mathfrak{sl}_{c},$ since a nilpotent element $B\in \mathfrak{gl}_{c},$ satisfies $tr(B)=0$ and is always in $\mathfrak{sl}_{c}.$  
The cases $c=2,3$ were treated in \cite[\S 5.1 \& \S 7]{Ngo}; for $c=2,$ $\mathcal{N}(n,2)$ is irreducibile of dimension $n+1,$ by \cite[Proposition 5.1.1]{Ngo}, it follows that $\mathcal{N}(n,2,r)^{st}$ is irreducible of dimension $2r+n+1.$ Hence, accounting for the quotient under the $GL(V)$ action, one gets $\dim\Quot_{\mathbb{Y}}^{[p]}(2,r)=2r+n-3.$ Irreducibility follows by applying Proposition \ref{number} as in item ${\rm (i)}.$ Moreover, it has been proved in \cite[Theorem 5.2.8]{Ngo}, that $\mathcal{N}(n,2,r)$ is normal, hence reduced. Then, also $\mathcal{N}(n,2,r)^{st}$ is reduced, and since the categorical quotient of a reduced scheme is also reduced, one obtains that $\Quot_{\mathbb{Y}}^{[p]}(2,r)$ is reduced.

For $c=3$ one can apply the same idea, as above, using the fact that $\mathcal{N}(n,3)$ is irreducible of dimension $2n+4$ as given by  \cite[Theorem 7.1.2]{Ngo}.

\end{itemize}
\vspace{0.3cm}

\end{proof}

\vspace{0.5cm}

We finish this section by adding another result on the connectedness of the punctual Quot scheme of points $\Quot_{\mathbb{Y}}^{[p]}(c,r),$ on a given smooth affine variety $\mathbb{Y},$ for some values of $c$ and $r.$

\begin{theorem}\label{teo-conecd}
	$\Quot_{\mathbb{Y}}^{[p]}(c,c),$ is path connected for all $c\in \N$ and for all $n\in \N$, where $n=\dim \mathbb{Y}$.
\end{theorem}
\begin{proof}
	Let $X=(B_0,\ldots,B_{n-1},v_1,\ldots,v_c)\in \mathcal{U}(n,c,c)^{st}$. Note that, for a basis $\{e_j\}_{j=1}^c$, we have $Y= (0,\ldots,0,e_1,\ldots,e_c)\in\mathcal{U}(n,c,c)^{st}$. Now, among the vectors $v_1,\ldots,v_c$, we consider only those who are linearly independent and write $v_1,\ldots,v_k$ with $k\leq c$.
	
	Complete the list $v_1,\ldots,v_k$ to a basis for $V$, say $$v_1,\ldots,v_k,w_1,\ldots,w_{c-k}.$$ In particular, the vectors $w_1,\ldots,w_{c-k}$ are given by the action, of some polynomials in the matrices $B_0,\ldots,B_{n-1},$ on the vectors $v_1,\ldots,v_k$, i.e., 
	$$w_i=\sum B_0^{\alpha_1}\cdots B_{n-1}^{\alpha_{n-1}}v_j,$$
	for some $\alpha_1,\ldots,\alpha_{n-1}$.

	Let $g$ be the matrix that has the vectors
	$v_1,\ldots,v_k,w_1,\ldots,w_{c-k}$ as its columns, then $g\in GL(V).$ Therefore, acting with $g$ on $Y$ we obtain $$g\cdot(0,\ldots,0,e_1,\ldots,e_c)=(0,\ldots,0,v_1,\ldots,v_k,w_1,\ldots,w_{c-k}).$$
	
	Now, consider $\varphi:[0,1]\to \mathcal{U}(n,c,c)^{st}$, defined by	 
	$$\varphi(t)=(t B_0,\ldots,t B_{n-1},v_1\,\ldots,v_k,w_1(1-t)+v_{k+1}t,\ldots,w_{c-k}(1-t)+v_c t).$$ 
	To see that $\varphi(t)\in \mathcal{U}(n,c,c)^{st}$ for all $t\in [0,1]$ note that for $t=0$, $\varphi(t)$ is stable because $v_1,\ldots,v_k,w_1,\ldots,w_{c-k}$ is a basis for $V$. For $t\neq0,$ we have
	$$V=\langle v_1,\ldots,v_k,B_0^{\alpha_1}\cdots B_{n-1}^{\alpha_{n-1}}v_j \rangle,$$
	with $j\in \{1,\ldots,k\},$ since $X$ is stable. Thus 
	$$V=\langle v_1,\ldots,v_k,(tB_0)^{\alpha_1}\cdots (tB_{n-1})^{\alpha_{n-1}}v_j \rangle,$$
	and $\varphi(t)$ is stable for all $t$.
	
It is also clear that the commutation relations are preserved  and $t B_j$ is nilpotent for all $j.$ Hence every stable datum $X$ can be connected to $\varphi(1)=Y$ as desired.
\end{proof}

Using the same idea in the proof of Theorem \ref{teo-conecd}, above, we are able to conclude the following:

\begin{corollary}\label{cor-conecd}
	Let $n,c\in \N$ and $r\geq n$. Then $\Quot_{\mathbb{Y}}^{[p]}(c,r)$ is path connected.
\end{corollary}


\section*{ Acknowledments}

The first author would like to thank Marcos Jardim for the useful remarks about the first draft of this paper.

\end{document}